\documentclass[12pt]{article}

\usepackage{amsmath,amsthm,amsfonts}
\usepackage{epsfig}
\usepackage[mathscr]{eucal}
\usepackage{amssymb,latexsym}

\renewcommand{\baselinestretch}{1.5}

\newtheorem{theorem}{Theorem}

\newtheorem{corollary}{Corollary}

\newcommand{\keywords}[1]
{\begin{center}
\begin{minipage}{315.83pt}
\small
We obtain bivariate forms of Gumbel's, Fr\'echet's and Chung's linear inequalities for $P(S\ge u, T\ge v)$ in terms of the bivariate binomial moments $\{S_{i,j}\}$, $1\le i\le k, 1\le j\le l$ of the joint distribution of $(S,T)$. At $u=v=1$, the Gumbel and Fr\'echet bounds improve monotonically with non-decreasing $(k,l)$. The method of proof uses combinatorial identities, and reveals a multiplicative structure before taking expectation over sample points.

\noindent \emph{Keywords:}~{bivariate binomial moments, Gumbel's inequality, combinatorial identity, Bonferroni-type inequalities}
\end{minipage}
\end{center}
\normalsize }

\title{Bivariate Binomial Moments and Bonferroni-type Inequalities\footnote{To appear in {\it Methodology and Computing in Applied Probability.}} }

\renewcommand{\baselinestretch}{0.90}

\author{Qin Ding \footnote{ International Exchange Student, University of Sydney, 2015. {\it viviandqdq@hotmail.com }}\\
and \\
Eugene Seneta \footnote{\it eseneta@maths.usyd.edu.au} \\School of Mathematics and Statistics, FO7\\
 University of Sydney, NSW 2006, Australia\\
}

\date{}

\begin{document}

\maketitle

\renewcommand{\baselinestretch}{1.5}

\begin{abstract}

\end{abstract}

\keywords

\section{Introduction}\label{intro}

The paper of Hoppe and Seneta $(2012)$ gives  a simplified treatment of several well-known, and some less well-known bounds on the probability of a union of events. The treatment  uses  binomial moments of a general non-negative integer-valued random  variable $T$ on sample space $\{0,1,2, \cdots, n \}$ to derive inequalities for  $P(T\geq v)$.  Gumbel's Identity provides the link to the events setting, where $T$ is the number out of $n$ events which occur. 

 The present paper extends the univariate methodology of Hoppe and Seneta $(2012)$ to derive extensions to a bivariate setting of the nature of such inequalities, in particular, of Fr\'echet, Gumbel and Chung, and of their special univariate properties.  The focus is identities and inequalities for $P(S \geq u, T \geq v)$ for a pair of non-negative integer-valued  random variables $(S,T)$ on $\{0,1,2,\cdots, m \} \times \{1,2, \cdots, n \}$.

Suppose $\{A_1,A_2, \cdots, A_m\}$ are arbitrary events in a probability space $(\Omega, B, P)$, and

\begin{equation}\label{old1}
S_{k}=\sum_{i\in I_{k,m}}P(A_{i_1},A_{i_2}, ...,A_{i_k}), k\ge 1,
\end{equation}
with  $S_{0}=1$ by definition, where the set $I_{k,m}$ consists of all k-tuples $i$=$\{i_1,i_2, \cdots ,i_k\}$, where $1\le i_1< i_2< \cdots<i_k\le m$.
The quantities (\ref{old1}) are called (univariate) Bonferroni Sums. If $S(\omega)=\# \{j:1\le j\le m, \omega \in A_j\}$ is the number of events $A_1$, $A_2$,..., $A_m$ which occur at a sample point $\omega$, {\it Gumbel's Identity} (Gumbel, 1936,1937) expresses $S_{k}$ as the k-th binomial moment of the random variables $S$:
\begin{equation}\label{old2}
S_k = E{{S}\choose k},
\end{equation}

Approaches to the use of Bonferroni sums (\ref{old1}) for equalities and inequalities on $P(S = u)$ and $P(S\geq u)$  have from the earliest times (see for example Fr\'echet (1940), and Galambos and Simonelli (1996)), been imbedded to a greater or lesser extent, in the ``events" setting, using in this context  manipulations of probabilities of unions, intersections and complements of events.

The total separation of expression of equalities and inequalities in terms of binomial moments using identities in binomial coefficients,  from the  subsequent application to the ``events" setting and expression in terms of Bonferroni sums in that setting, seems  arguably a more foundational and direct approach.  The derivation of known univariate inequalities, both linear and quadratic, in this manner is presented in Hoppe and Seneta (2012).  For example the inequality derived  in Hoppe and Seneta $(2012)$, \S3.4,

\begin{equation}\label{old3}
P(S\ge 1) \le \frac{ E{\big\{}{m\choose k} - {{m-S}\choose k}{\big\}}}  
{ {{m-1}\choose {k-1}} }
= \frac{ {m\choose k} - {\overline S}_k}  
{ {{m-1}\choose {k-1}} } , \, k\ge 1
\end{equation}
is the expression in terms of binomial moments of Gumbel's $(1936)$ Inequality. Here $${{\overline S}_k} =E{{m-S}\choose k}$$  devolves  to $\sum_{i\in I_{k,m}}P({\bar A}_{i_1},{\bar A}_{i_2}, \cdots,{\bar A}_{i_k})$ in the ``events"  setting, and hence the numerator of (\ref{old3}) to $\sum_{i\in I_{k,m}}P(A_{i_1}\cup A_{i_2} \cup \cdots \cup A_{i_k})$.

The {\it bivariate analogue of Gumbel's Identity} is 
 \begin{equation}\label{old4}
S_{k,l}={E{{S}\choose k}{{T}\choose l}},
\end{equation}
where ``T" is the counting random variate for the events set $B_1,B_2, \cdots,B_n$ and

\begin{equation}\label{old5}
S_{k,l}= \Sigma' P(A_{i_1},A_{i_2}, \cdots, A_{i_k};B_{j_1},B_{j_2}, \cdots, B_{j_l}),
\end{equation}
where $\{A_1,A_2,\cdots,A_m\}$, $\{B_1,B_2,\cdots,B_n\}$ are two sets of events in a probability space $(\Omega, B, P)$, the summation $\Sigma'$ is over the index set $\{(i_1,i_2,\cdots,i_k;j_1,j_2,\cdots,j_l), 1\le i_1< i_2<\cdots<i_k\le m, 1\le j_1< j_2<\cdots<j_l\le n\}$.

A general form of Gumbel's Identity for any finite number of counting random variables $\{S, T, \cdots \}$ is proved by the method of Indicators in Galambos and Simonelli $(1996)$, Chapter \rm{V}, Section \rm{VI}, so we have (\ref{old4}) to hand, in application of our bivariate binomial moment theory, developed in terms of $S_{k,l}$ for {\it general} random variables $(S,T)$, to the ``two sets" context.

 Our ``events-free" methodology and representation in terms of a bivariate distribution $(S,T)$ is not more ``general" than  the ``events" setting, since given such a bivariate distribution it is always possible to construct a probability space and event sets, $\{A_1, A_2, \cdots, A_m\}, \{B_1, B_2,\cdots. B_n\}$ in it, for which $S,T$ are the ``counting" variables. 

The ``events" setting is primary in practical applications  in calculating {\it values}  of  the bivariate binomial moments $S_{i,j}$ for small $(i,j)$  via Bonferroni sums using the the generalized Gumbel Identity, when the bivariate distribution of $S,T$ is not known.

Our  methodology via combinatorial identities provides relatively simple proofs, and indicates direction of extension to more than two dimensions. 

In the bivariate ``events" setting, the initial fundamental results, including (\ref{old4}), for $P(S=u, T=v),\, P(S \geq u, T \geq v)$  are in the note of Meyer (1969), which leans on an earlier such result in Fr\'echet (1940), which itself has a convoluted history (see our \S  2). The work of  Chen and Seneta (1996) (2000) (see also the recent book  of Chen (2014))  uses to a large extent combinatorial identities, but nevertheless  relies on Meyer's (1969) results to express the bounds in terms of the Bonferroni-type sums (\ref{old5}).

Thus our first task in Section 2 of this paper is to generalize in a self-contained fashion the results of Meyer (1969)  for a pair of non-negative integer-valued  random variables $(S,T)$ on $\{0,1,2,\cdots, m \} \times \{0,1,2, \cdots, n \}$, using (\ref{old4}) as the {\it definition} of the $(k,l)$ binomial moment of the distribution. We shall need some of these results in the following sections.  

\section{Basic Identities in Terms of Bivariate Binomial Moments}

\subsection{ The Fundamental Theorem.}

\begin{theorem}
If $(S,T)$ is a pair of random variables on $\{0,1, \cdots,$m$\} \times \{0,1, \cdots,$n$\}$, write $P_{[u,v]}=P(S=u, T=v)$. Then for $m,n\ge u,v\ge 0$,
\begin{equation}\label{old6}
P_{[u,v]}=P(S=u, T=v)=\sum^{m+n}_{t=u+v} \sum_{i+j=t}(-1)^{t-(u+v)}{{i}\choose u}{{j}\choose v}S_{i,j},
\end{equation}
where
\begin{equation} \label{old7}
{S_{i,j}=E{{S}\choose i}{{T}\choose j}},  m,n\ge i,j\ge 0.
\end{equation}
\end{theorem}
\begin{proof}
Write for the ordinary bivariate probability generating function:
\begin{equation}\label{old8}
P(t,s)=\sum^{m}_{u=0} \sum^{n}_{v=0}t^{u}s^{v}P_{[u,v]}
\end{equation}
Then, eventually using (\ref{old7}): 
\begin{eqnarray*}
P(1+t,1+s) &=&{\sum^{m}_{u=0} \sum^{n}_{v=0}(1+t)^{u}(1+s)^{v}P_{[u,v]}} \\
&=&\sum^{m}_{u=0} \sum^{n}_{v=0}\sum^{u}_{i=0} \sum^{v}_{j=0} {{u}\choose i}{{v}\choose j}t^{i}s^{j}P_{[u,v]} \\
&=&\sum^{m}_{i=0} \sum^{n}_{j=0}\sum^{m}_{u=i} \sum^{n}_{v=j} {{u}\choose i}{{v}\choose j}t^{i}s^{j}P_{[u,v]} \\
&=&\sum^{m}_{i=0} \sum^{n}_{j=0}t^{i}s^{j}S_{i,j}, \end{eqnarray*} so that
\begin{eqnarray}
P(t,s) &=&{\sum^{m}_{i=0} \sum^{n}_{j=0}(t-1)^{i}(s-1)^{j}S_{i,j}}\nonumber\\
&=&\sum^{m}_{i=0} \sum^{n}_{j=0}\sum^{i}_{u=0}\sum^{j}_{v=0}{{i}\choose u}t^{u}(-1)^{i-u}{{j}\choose v}s^{v}(-1)^{j-v}S_{i,j}\nonumber \\ 
&=&\sum^{m}_{u=0} \sum^{n}_{v=0}{\sum^{m}_{i=u}\sum^{n}_{j=v}{{i}\choose u}{{j}\choose v}(-1)^{i+j-(u+v)}S_{i,j}}t^{u}s^{v}.\label{old9} 
\end{eqnarray}
By comparing the coefficients of $t^us^v$ for $P(t,s)$ in (\ref{old8})  and (\ref{old9}) , we obtain (\ref{old6}) for $P_{[u,v]}.$ 
\end{proof}

This expression (\ref{old6}) may also usefully be written as
\begin{equation}\label{old10}
P_{[u,v]}=\sum^{m}_{i=u}\sum^{n}_{j=v}{{i}\choose u}{{j}\choose v}(-1)^{i+j-(u+v)}S_{i,j},
\end{equation}
so that, by (\ref{old7}):
\begin{corollary}
From (\ref{old10}) for $m,n\ge u,v\ge 0$,
\begin{equation}\label{old11}
P_{[u,v]}=E{\bigg\{}{\bigg (} \sum^{m}_{i=u}{{{i}\choose u}(-1)^{i-u}}{S \choose i}{\bigg )} {\bigg(}\sum^{n}_{j=v}{{{j}\choose v}(-1)^{j-v}}{T\choose j}{\bigg )}{\bigg\}}.
\end{equation}
\end{corollary}

Theorem 1 expresses the bivariate distribution in terms of the bivariate moments $S_{i,j}$. 

Our equation (\ref{old6})  is  equation $(1)$ in Meyer's $(1969)$ pioneering paper. However, his equation $(1)$  is  in terms of the bivariate set-specific quantities (our (\ref{old5})  above), and refers to Fr\'echet's $(1940)$ booklet for proof in that set-specific context. Fr\'echet $(1940)$, pp.50-52, obtains the expression using ``the second theorem of Broderick", from Broderick $(1937)$, who develops his theory in the ``symbolic" fashion akin to Bonferroni's (1936). This symbolic fashion is later present in Chapter 2, Section 6, of the classic book of Riordan $(1958),$  but in terms only of univariate binomial moments. 

We thought it time to give a simple, direct and non-symbolic proof as above, which is still in the spirit of Riordan $(1958)$, for completeness,  since Theorem 1 is a foundation for Theorem 2 below,  whose proof  is extremely condensed in Meyer (1969).

However, the following argument for Theorem 1, due partly  to Sibuya, and to Galambos and Xu (1995) where Sibuya is acknowledged, is notable because of its development in terms of combinatorial identities: 

Since 
\begin{equation}\nonumber
(1-1)^S=\sum_{i=0}^S (-1)^i {S\choose i}=\sum_{i=0}^m (-1)^i {S\choose i},
\end{equation}
\begin{equation}\label{star1}
E((1-1)^S)=(1-1)^0 P(S=0)=\sum_{i=0}^m (-1)^i E{S\choose i},
\end{equation}
where $(1-1)^0$ is interpreted as 1; and similarly for $P(T=0)$.

Now, since $(1-1)^{S+T}=(1-1)^S(1-1)^T$,
\begin{equation}\nonumber
\begin{split}
(-1)^t {S+T \choose t}&=\sum_{i=0}^t (-1)^i {S\choose i}(-1)^{t-i} {T\choose t-i}\\
&=\sum_{i+j=t} (-1)^t {S\choose i} {T\choose j},
\end{split}
\end{equation}
so
\begin{equation}\nonumber
\sum^{m+n}_{t=0} (-1)^t {S+T \choose t}=\sum^{m+n}_{t=0}\sum_{i+j=t} (-1)^t {S\choose i} {T\choose j},
\end{equation}
and taking expectation and using (\ref{star1}) with $S+T$ replacing $S$, we have
\begin{equation}\nonumber
P(S+T=0)=\sum^{m+n}_{t=0}\sum_{i+j=t} (-1)^t S_{i,j}.
\end{equation}
Since $P(S+T=0)=P(S=0, T=0)$, we have (\ref{old6}) of our Theorem 1, at $u=v=0$. Theorem 1 for general $u$ and $v$ follows from the special case $u=v=0$ by the reduction method encapsulated in Corollary $2.1$ of Galambos and Xu $(1995)$, see also Galambos and Simonelli $(1996)$.

Once univariate {\it inequalities} are available for $P(S+T=0)=P(S=0, T=0)$, the reduction method leads to inequalities for the general case $P(S=u, T=v)$. The reduction method applied to the classic univariate Bonferroni Inequalities leads to the inequalities of our Theorem 3, which are Meyer's version of extension to the bivariate case. 

\subsection {Theorems 2 and 3.}
The Corollary above, and  first part of the proof of Theorem2 below initiate  our methodology of working in terms of identities and inequalities  in terms of combinatorial quantities  applied to sample points of the pair of random variables $\{S,T\}$ before  using the linear property of expectation to express in terms of bivariate moments of their joint distribution.

Our method of proof is akin to, but distinct from, the method of indicator  functions which uses products of indicator functions before taking expectation. The methodof indicatore functions is concisely exposited at the beginning of Galambos and Xu (1995).

\begin{theorem}
Put
\begin{equation*}
P_{(u,v)}=P(S\ge u, T\ge v),
\end{equation*}
then:
\begin{equation}\label{old13}
P_{(u,v)}=\sum^{m+n}_{t=u+v} \sum_{i+j=t}(-1)^{i+j-(u+v)}{{i-1}\choose u-1}{{j-1}\choose v-1}S_{i,j},
\end{equation}
where $S_{i,j}$ is given by (\ref{old7}), and, further
\begin{equation}\label{old14}
S_{i,j}=\sum^{m}_{u=i} \sum^{n}_{v=j}{{u-1}\choose i-1}{{v-1}\choose j-1}P_{(u,v)}.
\end{equation}
\end{theorem}

\begin{proof}
From(\ref{old11}) :
\begin{eqnarray*}
P_{(u,v)}&=&\sum^{m}_{y=u}\sum^{n}_{z=v}P_{[y,z]}\\
&=& \sum^{m}_{y=u} \sum^{n}_{z=v}E{\bigg\{}{\bigg (} \sum^{m}_{i=y}{{{i}\choose y}(-1)^{i-y}}{S \choose i}{\bigg )} {\bigg(}\sum^{n}_{j=z}{{{j}\choose z}(-1)^{j-z}}{T\choose j}{\bigg )}{\bigg\}}\\
&=& E{\bigg\{}{\bigg (}\sum^{m}_{y=u} \sum^{m}_{i=y}{{{i}\choose y}(-1)^{i-y}}{S \choose i}{\bigg )} {\bigg(} \sum^{n}_{z=v}\sum^{n}_{j=z}{{{j}\choose z}(-1)^{j-z}}{T\choose j}{\bigg )}{\bigg\}}\\
&=&
E{\bigg\{}{\bigg (}\sum^{m}_{i=u} (-1)^{i-u} {{i-1}\choose {u-1}}{S \choose i}{\bigg )} {\bigg(} \sum^{n}_{j=v} (-1)^{j-v} {{j-1}\choose {v-1}}{T\choose j}{\bigg )}{\bigg\}}\end{eqnarray*}
where the last line follows from Combinatorial Identity 2 in our Section 6. 

Now, the  right hand side of (\ref{old14}), using ${d \choose k} = {{d-1}\choose k}  + {{d-1}\choose {k-1}} $ (Identity 1,  Section 6),

\begin{eqnarray*}
\sum^{m}_{u=i} \sum^{n}_{v=j}{{u-1}\choose i-1}{{v-1}\choose j-1}P_{(u,v)} &=&
\sum^{m}_{u=i} \sum^{n}_{v=j} {\bigg (} {{u}\choose i}-{{u-1}\choose i} {\bigg )} {\bigg (} {{v}\choose j}-{{v-1}\choose j}{\bigg )} P_{(u,v)} \\
&=&\sum^{m}_{u=i} \sum^{n}_{v=j} {{u}\choose i} {{v}\choose j} P_{[u,v]}=S_{i,j}.\end{eqnarray*} 
\end{proof} 

\begin{theorem}
For any non-negative integer k,

\begin{equation}\label{old15}
\sum^{u+v+2k+1}_{t=u+v}\sum_{i+j=t}g(i,j;t) \le P_{(u,v)} \le \sum^{u+v+2k}_{t=u+v}\sum_{i+j=t}g(i,j;t),
\end{equation}
where $g(i,j;t)=(-1)^{t-(u+v)}{{i-1}\choose u-1}{{j-1}\choose v-1}S_{i,j}$.
\end{theorem}

This expresses Bonferroni's Inequalities in our slightly generalized setting of a bivariate distribution of random variables $(S,T)$ and their binomial bivariate moments. This Theorem 3, which generalizes to the bivariate setting Theorem 3 in Hoppe and Seneta (2012),  needs no separate proof, since  Meyer's (1969) proof is completely appropriate. 

The Sobel-Uppuluri-Galambos Inequalities sharpen the Bonferroni Inequalities in the classic single-set context by adding (for the lower bound) and subtracting (for the upper bound) another binomial moment term. These are extended  to the context of distribution of a single general integer random variable $T$ on $\{0,1,\cdots,n\}$ by Hoppe and Seneta $(2012)$, Section 3.2. The extension of analogous bounds in this manner for the joint distribution of two general random variables in terms of $S_{k,l}$ is already contained in Chen and Seneta $(1996) ,(2000).$ Consequently we now proceed to the two-dimensional generalization of the less-known one-dimensional results in Hoppe and Seneta (2012), and do so, as in the one-dimensional case, by using combinatorial methods on sample points.

Our present paper, foreshadowed in Hoppe and Seneta (2012), was also stimulated by the appearance of M\'adi-Nagy and Pr\'ekopa (2015), which itself is partly motivated by that of Galambos and Xu (1995).  We shall say more in our \S 3.5

\section {Generalized Bivariate Fr\'echet and Gumbel Inequalities}

In the generalized univariate setting of Hoppe and Seneta $(2012)$, these inequalities for $P(T\ge 1)$ were expressed in terms of $S_{k}=E{{T}\choose k}$, $k\ge 0$, and ${\overline S_{k}}=E{{n-T}\choose k}.$

For the bivariate setting, recall from (\ref{old4}) that ${S_{k,l}=E{{S}\choose k}{{T}\choose l}}$ is an obvious 
analogue of  $S_{k}=E{{T}\choose k}$.

The bivariate analogue of ${\overline S_{k}}=E{{n-T}\choose k}$ is then
\begin{equation}\label{old17}{\overline S_{k,l}}={{m}\choose k}E{{n-T}\choose l}+{{n}\choose l}E{{m-S}\choose k}-E{{m-S}\choose k}{{n-T}\choose l}.
\end{equation}

To see the reason for this, and it is central to our sequel,  it  is convenient to proceed in terms of the quantity 
\begin{equation} \label{old18}
{\bigg\{} {{m}\choose k}-{{m-S}\choose k} {\bigg\}} {\bigg\{} {{n}\choose l}-{{n-T}\choose l}{\bigg\}}
\end{equation}
so  that, 
imitating the proof of Fr\'echet's univariate Inequality in Hoppe and Seneta (2012), \S3.3:
\begin{eqnarray*}
{{{m}\choose k}{{n}\choose l} - {\overline S_{k,l}} } 
&=&\sum^{m}_{i=0}\sum^{n}_{j=0} {\bigg\{} {{m}\choose k}-{{m-i}\choose k} {\bigg\}} {\bigg\{} {{n}\choose l}-{{n-j}\choose l} {\bigg\}} P_{[i,j]}\\
&\le& \sum^{m}_{i=1}\sum^{n}_{j=1} {{m}\choose k} {{n}\choose l}P_{[i,j]} \\
&=&{{m}\choose k} {{n}\choose l} P(S\ge 1, T\ge 1).\end{eqnarray*}
Thus with the definition (\ref{old17})  of  $ {\overline S_{k,l}}$ we can write down:

\subsection{Fr\'echet's Bivariate  Inequality}
\begin{equation}\label{old19}
P(S\ge 1, T\ge 1) \ge \frac { {{m}\choose k} {{n}\choose l} - {\overline S_{k,l}}}  {{{m}\choose k}{{n}\choose l}}.
\end{equation}

Furthermore we have an ``expectation of  product" form:
\begin{equation}\label{old20}
{{m}\choose k}{{n}\choose l} - {\overline S_{k,l}}= E{\bigg\{} {\bigg\{} {{m}\choose k}-{{m-S}\choose k} {\bigg\}} {\bigg\{} {{n}\choose l}-{{n-T}\choose l}{\bigg\}} {\bigg\}}.
\end{equation}

Expressions of products of linear forms in  binomial coefficients  pervade our subsequent development, and are already evident  in our \S 1, for example in (\ref{old10}). They make it clear how our various bivariate equalities and  inequalities  generalize
from univariate, and allow generalization to multivariate. 

Before  proceeding we note that in (\ref{old19}) above, and  all the bounds below involving ${\overline S_{k,l}},$ can be expressed in the customary way in terms of  linear combinations of the bivariate moments $S_{i,j}$ on account of (\ref{old42}) in the sequel, according to which:
\begin{equation}\label{old21}
{\overline S_{k,l}} =  {{m}\choose k}{{n}\choose l} - \sum^{k}_{s=1} \sum^{l}_{r=1}  (-1)^{s+r} {{m-s}\choose k-s} {{n-r}\choose l-r}  S_{s,r}.
\end{equation}

\subsection{Gumbel's Bivariate Inequality}
\begin{equation}\label{old22}
P(S\ge 1, T\ge 1) \le {{{m}\choose k}{{n}\choose l} - {\overline S_{k,l}} \over {{m-1}\choose k-1}{{n-1}\choose l-1}}.
\end{equation}

\begin{proof}
When $i\ge 1$,  ${{m-1}\choose k}\ge {{m-i}\choose k}$, so from Combinatorial Identity 1 (in Section 6),
\begin{eqnarray*}
{{m}\choose k}-{{m-i}\choose k}&=&{{m-1}\choose k}+{{m-1}\choose k-1}-{{m-i}\choose k}\\
&\ge& {{m-1}\choose k-1}.\end{eqnarray*}
Therefore, from the proof of Fr\'echet's Bivariate Inequality above, 
\begin{eqnarray*}
{{m}\choose k}{{n}\choose l} - {\overline S_{k,l}}
&=&\sum^{m}_{i=0}\sum^{n}_{j=0} {\bigg\{} {{m}\choose k}-{{m-i}\choose k}{\bigg\}} {\bigg\{} {{n}\choose l}-{{n-j}\choose l} {\bigg\}}P(S=i, T=j)\\
&\ge& \sum^{m}_{i=1}\sum^{n}_{j=1} {{m-1}\choose k-1} {{n-1}\choose l-1}P(S=i, T=j) \\
&=&{{m-1}\choose k-1} {{n-1}\choose l-1} P(S\ge 1, T\ge 1),\end{eqnarray*}
which reduces to (\ref{old22}).
\end{proof}

Note that the numerator in both cases (\ref{old19}) and (\ref{old22}) is (\ref{old20}), as is the case also in the right hand side of (\ref{old30}) below.

\subsection{Monotonicity and Concavity of Fr\'echet's Bounds}
Put
\begin{equation}\label{old23}
L_{k,l} = \frac{ {\big\{} {{m}\choose k}-{{m-S}\choose k} {\big\}} {\big\{} {{n}\choose l}-{{n-T}\choose l} {\big\}} } {{{m}\choose k}{{n}\choose l}}.
\end{equation}
Fix $l$, then
\begin{eqnarray}
L_{k,l}-L_{k+1,l} &= &\frac{ {{n}\choose l}-{{n-T}\choose l} }  { {{n}\choose l} } {\bigg\{} {\bigg (} 1- \frac{ {{m-S}\choose k} } { {{m}\choose k} } {\bigg )} -{\bigg (} 1- \frac{ {{m-S}\choose k+1} } { {{m}\choose k+1} } {\bigg )} {\bigg\}}  \nonumber \\
& = &\frac{ {{n}\choose l}-{{n-T}\choose l} } {{{n}\choose l}} {\bigg\{} \frac{(m-S)!(m-k-1)!} {(m-S-k-1)!m!} - \frac{(m-S)!(m-k)!} {(m-S-k)!m!}{\bigg\}} \nonumber\\
& =&\frac{ {{n}\choose l}-{{n-T}\choose l} } {{{n}\choose l}}  \frac{(m-S)!(m-k-1)!} {(m-S-k-1)!m!}  {{\bigg (} 1- \frac{m-k} {m-S-k} {\bigg )} }.\label{old24}
\end{eqnarray}

Since ${1-{{m-k}\over {m-S-k}}\le 0}$, we get 
\begin{equation}\label{old25}
L_{k,l}\le L_{k+1,l}.
\end{equation}

Taking expectations in both sides of (\ref{old23})  and (\ref{old25}), we have Fr\'echet's bound equal to $E(L_{k,l})$. So Fr\'echet's bound  increases with increasing $k$. Similarly,  Fr\'echet's bound increases with increasing $l.$

To prove concavity, use (\ref{old24}),
\begin{equation}\nonumber
\begin{split}
L_{k,l}-L_{k+1,l}&= \frac{ {{n}\choose l}-{{n-T}\choose l} } {{{n}\choose l}} \frac {(m-S)!(m-k-1)!}{ (m-S-k-1)!m!}  {\bigg (} 1- \frac{m-k} {m-S-k} {\bigg )} \\
&=\frac{ {{n}\choose l}-{{n-T}\choose l} }  {{{n}\choose l}}  \frac{(m-S)!(m-k-2)!} {(m-S-k-2)!m!}  \frac{m-k-1}{m-S-k-1} \frac{-S} {m-S-k}\\
&=\frac{{{n}\choose l}-{{n-T}\choose l} } {{{n}\choose l}}  \frac{(m-S)!(m-k-2)!} { (m-S-k-2)!m!}  \frac{m-k-1} {m-S-k-1} \frac{m-S-k-1} {m-S-k} \frac{-S} {m-S-k-1}\\
&=\frac{ {{n}\choose l}-{{n-T}\choose l} } {{{n}\choose l}}  \frac{(m-S)!(m-k-2)!} {(m-S-k-2)!m!}  \frac{m-k-1}{m-S-k} {\bigg (} 1-\frac{m-k-1} {m-S-k-1} {\bigg )} \\
& = \frac{m-k-1} {m-S-k} (L_{k+1,l}-L_{k+2,l}).
\end{split}
\end{equation}

Since for $S=0$, $L_{k,l}-L_{k+1,l}=0$ for all $k\ge 0$, and for $1\le S\le n$, ${{m-k-1}\over {m-S-k}}\ge 1$ and $L_{k,l}-L_{k+1,l} \le 0$ for all $k\ge 0$, so

$$L_{k,l}-L_{k+1,l} \le L_{k+1,l}-L_{k+2,l}.$$

Therefore, Fr\'echet's bounds is concave in $k$.  We can  prove the concavity with respect to the second parameter $l$ similarly.

\subsection{Monotonicity and Convexity of Gumbel's Bounds}
Gumbel's inequality applies equally to random variables $U=m-S$ and $V=n-T$. So,
\begin{equation}\label{old26}
P(U\ge 1, V\ge 1) \le {E {\big\{} {{m}\choose k}-{{S}\choose k} {\big\}} {\big\{} {{n}\choose l}-{{T}\choose l} {\big\}} \over {{m-1}\choose k-1}{{n-1}\choose l-1}}.
\end{equation}
Denote  the right hand side of (\ref{old26}) as $G_{k,l}$. By a double use of Combinatorial Identity 3 (Section 6):
\begin{equation}\label{old27}
E{{S}\choose k}={{m}\choose k}-\sum^{m}_{i=k}{{i-1}\choose k-1} P(S<i),
\end{equation}
and from our Theorem 2, specifically (\ref{old14}):
\begin{equation}\label{old28}
E{{S}\choose k}{{T}\choose l}=\sum^{m}_{i=k}\sum^{n}_{j=l}{{i-1}\choose k-1}{{j-1}\choose l-1} P(S\ge i, T\ge j),
\end{equation}
we have
\begin{eqnarray*}
&{}& {{m}\choose k}{{n}\choose l}-{{m}\choose k}E{{T}\choose l}-{{n}\choose l}E{{S}\choose k}+E{{S}\choose k}{{T}\choose l}\\
&=&{{m}\choose k}{{n}\choose l}
-  {{m}\choose k} {\bigg\{} {{n}\choose l}-\sum^{n}_{j=l}{{j-1}\choose l-1} P(T<j) {\bigg\}}\\
& - &{{n}\choose l} {\bigg\{} {{m}\choose k}-\sum^{m}_{i=k}{{i-1}\choose k-1} P(S<i) {\bigg\}} \\
&+& \sum^{m}_{i=k}\sum^{n}_{j=l}{{i-1}\choose k-1}{{j-1}\choose l-1} \{1-P(S\ge i, T<j)-P(S< i, T\ge j)-P(S< i, T<j)\} \\
&= &\sum^{m}_{i=k}\sum^{n}_{j=l} {{i-1}\choose k-1}{{j-1}\choose l-1} \{ P(T<j)+P(S<i)-P(S\ge i, T<j)\cr &-&P(S< i, T\ge j)-P(S< i, T<j)\} \\
&= &\sum^{m}_{i=k}\sum^{n}_{j=l} {{i-1}\choose k-1}{{j-1}\choose l-1} P(S< i, T<j) .
\end{eqnarray*}

Therefore,
\begin{eqnarray*}
G_{k,l} &=&\frac{\sum^{m}_{i=k}\sum^{n}_{j=l} {{i-1}\choose k-1}{{j-1}\choose l-1} P(S< i, T<j)} {{{m-1}\choose k-1}{{n-1}\choose l-1}}\\
&=& \sum^{n}_{j=l} \frac{{{j-1}\choose l-1}} {{{n-1}\choose l-1}} \sum^{m}_{i=k} \frac{(i-1)(i-2)...(i-k+1) }{ (m-1)(m-2)...(m-k+1)} P(S< i, T<j).\end{eqnarray*}

Fix $l$, then  we have
\begin{equation}\label{old29}
G_{k,l}- G_{k+1,l}=\sum^{n}_{j=l} \frac{{{j-1}\choose l-1}} {{{n-1}\choose l-1}} \sum^{m}_{i=k} \frac{(i-1)(i-2)...(i-k+1)} {(m-1)(m-2)...(m-k+1)} {\bigg (}1-\frac{i-k} {m-k} {\bigg )} P(S< i, T<j).
\end{equation}
Since $k\le i\le m$, so $G_{k,l}\ge G_{k+1,l}.$ Thus, $G_{k,l}$ is decreasing with increasing $k$.  We can get a similar conclusion for the second parameter $l$.

Put  $D_{k,l}=G_{k,l}- G_{k+1,l}$. Using (\ref{old29}),
$$D_{k,l}- D_{k-1,l}={\sum^{n}_{j=l} {{{j-1}\choose l-1}\over {{n-1}\choose l-1}} } \sum^{m}_{i=k-1} {(i-1)(i-2)...(i-k+2) \over (m-1)(m-2)...(m-k+1)} (m-i) {\bigg (} {{i-k+1}\over {m-k}}-1 {\bigg )}  P(S< i,T<j).$$
Since for $i=m$, $D_{k,l}- D_{k-1,l}=0$ for all $k$, and for $k-1\le i< m$, ${{i-k+1}\over {m-k}}-1\le 0$, so we obtain that $D_{k,l}- D_{k-1,l}\le 0$ for all $k$. So, Gumbel's bound is convex in $k$. We can  conclude similarly the convexity of Gumbel's bound with respect to the second parameter $l$.

\subsection{Comparison}
For fixed $k,l\ge 1$, the Fr\'echet's lower bound (\ref{old19}) and Gumbel's upper bound (\ref{old22}) for $P_{(1,1)}=P(S\ge1, T\ge 1)$, when expressed, using (\ref{old20}), in terms of bivariate moments $S_{i,j}$ are a linear combination of all such moments for $1\le i\le k$, $1\le j\le l$. Although we have shown the monotonicity of these bounds with non-decreasing $(k,l)$, it is appropriate to compare these bounds, for small fixed $(k,l)$, with existing bounds, some of which are optimal in a linear sense.

The linear upper bound 
\begin{equation}\label{c1}
P_{(1,1)}\le S_{1,1}- \frac 2 {n} S_{1,2}- \frac 2 {m} S_{2,1} + \frac 4 {mn} S_{2,2}
\end{equation}
due to Galambos and Xu $(1993)$ is optimal in terms of $\{S_{1,1},S_{1,2},S_{2,1},S_{2,2}\}$ in several senses (see also Seneta and Chen $(1996)$).

The bivariate Gumbel upper bound (\ref{old22}) at $k=l=2$, reads
\begin{equation}\label{c2}
P_{(1,1)}\le S_{1,1}- \frac 1 {n-1} S_{1,2}- \frac 1 {m-1} S_{2,1} + \frac 1 {(m-1)(n-1)} S_{2,2}
\end{equation}
is weaker, in an obvious sense, excepts when $m=n=2$. For the example on p.103 of Chen and Seneta $(1995)$, where $m=n=3$, the right-hand sides of (\ref{c1}) and (\ref{c2}) are respectively 0.888 and 0.891.

A lower bound (Chen and Seneta $(1995)$) for $P_{(1,1)}$ analogous to (\ref{c1}) is 
\begin{equation}\label{c3}
P(S\ge 1, T\ge 1)\ge \frac {4S_{1,1}} {(a+1)(b+1)} - \frac {4S_{1,2}} {b(a+1)(b+1)} - \frac {4S_{2,1}} {a(a+1)(b+1)} + \frac {4S_{2,2}} {ab(a+1)(b+1)} 
\end{equation}
where $a$ and $b$ are integers, providing $m-2a-1\le 0$, $n-2b-1\le 0$. Thus, if we choose $a=m-1, b=n-1$, we obtain
\begin{equation}\label{c4}
P(S\ge 1, T\ge 1)\ge \frac {4} {mn} S_{1,1} - \frac {4} {mn(n-1)} S_{1,2}- \frac {4} {mn(m-1)}S_{2,1} + \frac {4} {mn(m-1)(n-1)} S_{2,2},
\end{equation}
which is Fr\'echet's bivariate lower bound (\ref{old19}).

Note that both (\ref{c1}) and (\ref{c3}) were originally obtained  using inequalities for linear functions of combinatorial quantities on bivariate sample points.

Very recently M\'adi-Nagy and Pr\'ekopa $(2015)$, in somewhat the same tradition as Galambos and Xu $(1993)$, $(1995)$, have looked for coefficients $c_{s,t}, d_{s,t}$, $s+t\le w$ to satisfy
\begin{equation}\label{c5}
\sum_{s=0}^w \sum_{t=0}^{w-s} c_{s,t} S_{s,t} \le r(u,v; m,n) \le \sum_{s=0}^w \sum_{t=0}^{w-s} d_{s,t} S_{s,t} 
\end{equation}
where $w\le min(m,n)$, and $r(u,v; m,n)=P(S=u, T=v)$ or $r(u,v; m,n)=P(S\ge u, T\ge v)$. The authors use a linear programming approach for functions defined on the joint sample space of $(S,T)$ to obtain, after taking expectation, linear bounds for $r(u,v; m,n)$. These linear bounds are optimal in a certain sense.

The constraint in (\ref{c5}) that $s+t\le w$ is justified by the authors (p.25):`` usually the probabilities of intersections are given up to a certain number of counts, hence the multivariate moments up to a certain total order can be calculated.'' However, this makes a little awkward  comparison where the constraint on included bivariate binomial moments is $1\le s \le k, 1\le t \le l$: for example, comparing $w=4$, with $k=l=2$. Nevertheless, the upper bound (3.7) of M\'adi-Nagy and Pr\'ekopa $(2015)$,
\begin{equation}\label{c6}
P_{(1,1)}\le min {\bigg (} S_{1,1}- \frac 2 {mn} S_{1,2}-\frac 2 {m} S_{2,1}, S_{1,1}- \frac 2 {n} S_{1,2}-\frac 2 {mn} S_{2,1} {\bigg )}
\end{equation}
is clearly better than the bound 
\begin{equation}\nonumber
P_{(1,1)}\le min {\bigg (} S_{1,1}- \frac 1 {m-1} S_{2,1}, S_{1,1}- \frac 1 {n-1} S_{1,2} {\bigg )},
\end{equation}
obtained from the cases $k=1,$ $l=2$ and $k=2,$ $l=1$ of the Gumbel upper bound (\ref{old22}).

\section{Fr\'echet-type and Gumbel-type Inequalities}
\begin{theorem}
For $1\le s,k\le m$ and $1\le t,l\le n$, 
\begin{equation}\label{old30}
1-{{\overline S_{k,l}} \over {{m-s+1}\choose k}{{n-t+1}\choose l}} \le P(S\ge s, T\ge t) \le {{{m}\choose k}{{n}\choose l} - {\overline S_{k,l}} \over {\big (} {{m}\choose k}-{{m-s}\choose k} {\big )} {\big (} {{n}\choose l}-{{n-t}\choose l} {\big )}}.
\end{equation}
\end{theorem}
\begin{proof}

For the lower bound, from (\ref{old20}),

\begin{eqnarray*}
&{}&{{m}\choose k}{{n}\choose l} - {\overline S_{k,l}}\\
&=&\sum^{m}_{i=0}\sum^{n}_{j=0} {\bigg (} {{m}\choose k}-{{m-i}\choose k} {\bigg )} {\bigg (} {{n}\choose l}-{{n-j}\choose l} {\bigg )} P_{[i.j]} \\
&\le & \sum^{s-1}_{i=0}\sum^{t-1}_{j=0} {\bigg (} {{m}\choose k}
-{{m-s+1}\choose k} {\bigg )}  {\bigg (}{{n}\choose l}
-{{n-t+1}\choose l} {\bigg )}P_{[i.j]} \\
&+& \sum^{s-1}_{i=0}\sum^{n}_{j=t} {\bigg (} {{m}\choose k}-{{m-s+1}\choose k} {\bigg )} {{n}\choose l} P_{[i.j]} 
+\sum^{m}_{i=s}\sum^{t-1}_{j=0} {{m}\choose k} {\bigg (} {{n}\choose l}-{{n-t+1}\choose l} {\bigg )} P_{[i.j]}  \\
&+& \sum^{m}_{i=s}\sum^{n}_{j=t} {{m}\choose k}{{n}\choose l} P_{[i.j]} \\  
&\le &{{m}\choose k}{{n}\choose l}- \sum^{s-1}_{i=0}\sum^{t-1}_{j=0} {\bigg [}{{m-s+1}\choose k}{{n}\choose l}+{{m}\choose k}{{n-t+1}\choose l}\\&-&{{m-s+1}\choose k}{{n-t+1}\choose l}{\bigg ]}P_{[i.j]} - \sum^{s-1}_{i=0}\sum^{n}_{j=t} {{m-s+1}\choose k} {{n-t+1}\choose l}P_{[i.j]} \\
&- &\sum^{m}_{i=s}\sum^{t-1}_{j=0} {{m-s+1}\choose k} {{n-t+1}\choose l} P_{[i.j]} \\
& \le &{{m}\choose k}{{n}\choose l}
- \sum^{s-1}_{i=0}\sum^{t-1}_{j=0} {{m-s+1}\choose k}{{n-t+1}\choose l}P_{[i.j]} \\
&-& \sum^{s-1}_{i=0}\sum^{n}_{j=t} {{m-s+1}\choose k} {{n-t+1}\choose l} P_{[i.j]}
- \sum^{m}_{i=s}\sum^{t-1}_{j=0} {{m-s+1}\choose k} {{n-t+1}\choose l} P_{[i.j]} \\
&= &{{m}\choose k}{{n}\choose l} - {{m-s+1}\choose k} {{n-t+1}\choose l} (1- P(S\ge s, T\ge t)),
\end{eqnarray*}

which reduces to the left hand side of (\ref{old30}).

Next, again using (\ref{old20}) as initial step,
\begin{eqnarray*}
{{m}\choose k}{{n}\choose l} - {\overline S_{k,l}}
& \ge &\sum^{m}_{i=s}\sum^{n}_{j=t} {\bigg (} {{m}\choose k}-{{m-i}\choose k} {\bigg )} {\bigg (} {{n}\choose l}-{{n-j}\choose l} {\bigg )} P_{[i.j]} \\
& \ge &\sum^{m}_{i=s}\sum^{n}_{j=t} {\bigg (} {{m}\choose k}-{{m-s}\choose k} {\bigg )} {\bigg (} {{n}\choose l}-{{n-t}\choose l} {\bigg )} P_{[i.j]} \\
&=& {\bigg (} {{m}\choose k}-{{m-s}\choose k} {\bigg )} {\bigg (} {{n}\choose l}-{{n-t}\choose l} {\bigg )} P(S\ge s, T\ge t),\end{eqnarray*}
which reduces to the right hand side of (\ref{old30}). 
\end{proof}

\section{ Monotonicity and Convexity in Chung-type Bounds}

Define 

\begin{equation}\label{old31}
A^{(s,t)}_{k,l} =\frac
{\sum^{k}_{i=s} \sum^{l}_{j=t}  (-1)^{i+j-(s+t)}  {{i-1}\choose i-s} {{m-i}\choose k-i} {{j-1}\choose j-t} {{n-j}\choose l-j} S_{i,j} }  { {{m-s}\choose k-s} {{n-t}\choose l-t}}.
\end{equation}
We follow  Hoppe and Seneta (2012), p.283, from their equation $(43)$, to  prove that $A^{(s,t)}_{k,l}$ is monotone and convex in $k$ and $l$.  For $k = s,  s+1, \ldots, m-1$,

\begin{eqnarray*}
A^{(s,t)}_{k,l} - A^{(s,t)}_{k+1,l} &=&
{\sum^{l}_{j=t} (-1)^{j-t}  {{j-1}\choose j-t} {{n-j}\choose l-j} \over  {{n-t}\choose l-t}  } {\sum^{k+1}_{i=s} (-1)^{i-s}  {{i-1}\choose i-s} 
{\bigg [} { {{m-i}\choose k-i} \over  {{m-s}\choose k-s}  }-  {{{m-i}\choose k+1-i} \over {{m-s}\choose k+1-s}}} {\bigg ]}  S_{i,j} \\
&=& {\sum^{l}_{j=t} (-1)^{j-t}  {{j-1}\choose j-t} {{n-j}\choose l-j} \over  {{n-t}\choose l-t}} 
\sum^{k+1}_{i=s} (-1)^{i-s-1} {s\over {m-s} } { {{i-1}\choose  i-s-1}{{m-i}\choose k+1-i} \over {{m-s-1}\choose k-s} } S_{i,j}\\
&=& {s\over {m-s}} A^{(s+1,t)}_{k+1,l}\ge 0.\end{eqnarray*}

Therefore, using monotonicity in $k$, we have

$${A^{(s,t)}_{k,l} - A^{(s,t)}_{k+1,l} = {s\over {m-s}} A^{(s+1,t)}_{k+1,l}\ge {s\over {m-s}} A^{(s+1,t)}_{k+2,l}=A^{(s,t)}_{k+1,l} - A^{(s,t)}_{k+2,l}
}$$ 
proving convexity in $k$. Analogous results in $l$ follow similarly.

Again, the proof above  is essentially the univariate one from Hoppe and Seneta (2012), p. 283, because of the  product structure of (\ref{old31}), before  taking expectation:

\begin{equation}\label{old32}
A^{(s,t)}_{k,l} =
E{\bigg\{} \frac { \sum^{k}_{i=s}  (-1)^{i-s}  {{i-1}\choose i-s} {{m-i}\choose k-i}{S\choose i} }   {{{m-s}\choose k-s}} \frac{ \sum^{l}_{j=t}  (-1)^{j-t}  {{j-1}\choose j-t} {{n-j}\choose l-j}{T\choose j} }   {{{n-t}\choose l-t}} {\bigg\}}.
\end{equation}

When $k=m$ and $l=n$, we see from Theorem 2, and specifically from (\ref{old13}), that 
 $A^{(s,t)}_{m,n}=P(S\ge s, T\ge t),$ and since, for example, $A^{(s,t)}_{m-1,n} \geq A^{(s,t)}_{m,n}$, by monotonicity, we have  
 
\begin{equation}\label{old33}
P(S\ge s, T\ge t)\leq A^{(s,t)}_{m-1,n}.
\end{equation}
So such (Chung-type)  upper bounds are of the nature of the Gumbel-type upper  bounds as on the right-hand side of (\ref{old30}). Note that the bound in (\ref{old33})  is already expressed as a linear combination  of  bivariate binomial moments $S_{i,j}.$

\section{Combinatorial Identities}
Note that here and throughout this paper, for any real number $d$ and integer $r>0, {{d}\choose r}={d(d-1)...(d-r+1) \over r!}$. If $r=0,{{d}\choose r}=1.$ If $m$ and $r$ are positive integers, and $r> m$, then ${{m}\choose r}=0$.

Identities 1 to 4 occur,  with the same numbering, and are proved,  in Hoppe and Seneta (2012), p. 273. They are stated here for the reader's convenience, since all are used  in this paper.

{\bf Identity 1} (Extended Pascal's Identity) For any real number $d$ and integer $k \geq 1$,
\begin{equation*}
{d \choose k } = {{d-1} \choose {k}}  + {{d-1} \choose {k-1}}.
\end{equation*}

{\bf Identity 2} For integers $n \geq 1, k \geq 0,$
\begin{equation*}
\sum_{x=0}^{k} (-1)^x {n \choose x} = (-1)^k {{n-1} \choose k}. 
\end{equation*}

{ \bf Identity 3} For $k \geq 1, n, k$ integers
\begin{equation*}
{n \choose k} = \sum_{x=k}^n {{x-1}\choose {k-1}}= \sum_{x=0}^n {{x-1}\choose {k-1}}.
\end{equation*}

{\bf Identity 4} For $ n \geq k \geq 1, r \geq 1,$
\begin{equation*}
{n \choose k} = \sum_{j=1}^{r-1} {{n-j}\choose {k-1}} \ + \ {{n-r+1} \choose k}.
\end{equation*}

{\bf Identity 5}
\begin{equation}\label{old38}
{{n-T}\choose l}   = \sum_{r=0}^{l} (-1)^r {{n-r}\choose{l -r}}{ T \choose r}, T=0,1,2, ...n. 
\end{equation}

\begin{proof}
When $l=0,1$, we  see that (\ref{old38}) holds for every $n\ge l$. Assume when $l=k$, (\ref{old38})  holds for every $n\ge l$. By this induction assumption, it holds for $n-1,n-2,\cdots,k+1,k$. Using Identity 1, iterating back,
\begin{eqnarray*}
{{n-T}\choose k+1} &= &{{n-T-1}\choose k} +{{n-T-1}\choose k+1} \\
&=&{{n-T-1}\choose k} +{{n-T-2}\choose k} +...+  {{k+1}\choose k} +{{k+1}\choose k+1} \\
&= & \sum_{N=T}^{n-1}\sum_{r=0}^{k} (-1)^r {{N-r}\choose k-r} {{T}\choose r}\end{eqnarray*}
\noindent by the induction hypothesis; and now, after an exchange of order of summation, by Identity 4: 
\begin{eqnarray*}
\qquad \qquad &=& \sum_{r=0}^{k} (-1)^r  \sum_{N=T}^{n-1}{{N-r}\choose k-r} {{T}\choose r}\\
&=& \sum_{r=0}^{k} (-1)^r {{n-r}\choose k+1-r} {{T}\choose r} + \sum_{r=0}^{k} (-1)^{r+1} {{T-r}\choose k+1-r}{{T}\choose r} \\
&=& \sum_{r=0}^{k} (-1)^r {{n-r}\choose k+1-r} {{T}\choose r} + \sum_{r=0}^{k} (-1)^{r+1}  {{T}\choose k+1}{{k+1}\choose r}.\end{eqnarray*}
So the only thing we need to prove now is that
$$ \sum_{r=0}^{k} (-1)^{r+1}  {{T}\choose k+1}{{k+1}\choose r} = (-1)^{k+1}  {{T}\choose k+1}.$$
From Identity 2 above,
\begin{equation*}
\sum_{r=0}^{k} (-1)^{r+1} {{k+1}\choose r} = (-1)^{k+1}.
\end{equation*}
\noindent  So  we have  the desired conclusion, by induction. 
\end{proof}

We now apply  Identity 5, to express ${\overline S_{l}},{\overline S_{k,l}}$ as linear functions of binomial moments.

 \leftline{\bf 1. Univariate case} 

\bigskip

\noindent Taking expectations in both sides of (\ref{old38}) , we have that,
$$ E{{n-T}\choose l} = \sum_{r=0}^{l} (-1)^r {{n-r}\choose l-r} E{{T}\choose r},$$
which is 
\begin{equation}\label{old40}
{\overline S_{l}}=  \sum_{r=0}^{l} (-1)^r {{n-r}\choose l-r} S_{r} .
\end{equation}

 \leftline{\bf 2. Bivariate case} 
 
 \bigskip

\noindent Using (\ref{old38}), we have 

\begin{equation}\label{old41}
{{m-S}\choose k} {{n-T}\choose l} = \sum_{s=0}^{k} \sum_{r=0}^{l} (-1)^{s+r}  {{m-s}\choose k-s} {{n-r}\choose l-r}  {{S}\choose s}{{T}\choose r} .
\end{equation}
Take expectations in both sides, we have
$$ E{{m-S}\choose k} {{n-T}\choose l} = \sum_{s=0}^{k} \sum_{r=0}^{l} (-1)^{s+r}  {{m-s}\choose k-s} {{n-r}\choose l-r} S_{s,r} .$$
Thus, 
\begin{eqnarray}
{\overline S_{k,l}}&=& E {{m-S}\choose k}{n \choose l} +E { {n-T}\choose l}{m \choose k} - E {{m-S}\choose k} {{n-T}\choose l} \nonumber\\
&= &\sum_{s=0}^{k} (-1)^{s}  {{m-s}\choose k-s} {{n}\choose l} S_{s} +\sum_{r=0}^{l} (-1)^{r}  {{n-r}\choose l-r} {{m}\choose k} S_{r} \nonumber \\
&- &\sum_{s=0}^{k} \sum_{r=0}^{l} (-1)^{s+r}  {{m-s}\choose k-s} {{n-r}\choose l-r} S_{s,r} \nonumber\\
&=& {{m}\choose k}{{n}\choose l} - \sum^{k}_{s=1} \sum^{l}_{r=1}  (-1)^{s+r} {{m-s}\choose k-s} {{n-r}\choose l-r}  S_{s,r} \label{old42}.
\end{eqnarray}
\section {Conclusion}

Self-contained proofs of Meyer's (1969) results have been followed by the introduction 
 for the first time of  bivariate versions of the Fr\'echet, Gumbel and Chung inequalities, and demonstration of  their monotonicity and convexity properties. The method of proof has been via combinatorial identities, in a relatively simple  manner which departs from the usual ``events" setting for Bonferroni-type inequalities.
This has completed the study of bivariate versions of the known {\it linear} inequalities studied in Hoppe and Seneta (2012).  A study of bivariate extension of  the univariate {\it quadratic}  inequalities studied in that paper  is in progress.

\centerline{\bf Appendix}

We first provide an alternative rationale for using (\ref{old17}) as the ``correct" definition of ${\overline S_{k,l}}.$

As noted in our \S 1, in the univariate ``events" setting, the numerator on the right-hand side of Gumbel's inequality (\ref{old3}) for $P(S \geq 1)$ can be written as

\begin{equation*}
{{{m}\choose k} - {\overline S}_k} =\sum_{i\in I_{k,m}}P(A_{i_1}\cup A_{i_2} \cup ... \cup A_{i_k}).
\end{equation*}

This is also the numerator of the Gumbel-type upper bound for $P(S \geq s)$, from Hoppe and Seneta (2012), \S 5.2.
It is then plausible that in the bivariate generalization, in the ``events"  situation, the numerator of the Gumbel bound for $P(S \geq s,  T \geq t)$ which is just notation for $P(S \geq s \cap  T \geq t)$ should have numerator 
\begin{equation}\label{old44}
{{{m}\choose k}{{n}\choose l} - {\overline S_{k,l}} } = \sum_{ i\in I_{k,m}} \sum_{ j\in J_{l,n}} P((A_{i_{1}} \cup A_{i_{2}}\cup ...\cup A_{i_{k}})\cap ( B_{j_{1}} \cup B_{j_{2}}\cup ...\cup B_{j_{l}})).
\end{equation}

Taking this as the definition of ${\overline S_{k,l}}$,  probability manipulation of the right-hand side of  (\ref{old44}) gives
\begin{equation*}
 {\overline S_{k,l}} = \sum_{ i\in I_{k,m}} \sum_{ j\in I_{l,n}} [ P({\overline A_{i_{1}}}, ..., {\overline A_{i_{k}}}) + P({\overline B_{j_{1}}}, ..., {\overline B_{j_{l}}}) -P({\overline A_{i_{1}}}, ..., {\overline A_{i_{k}}};  {\overline B_{j_{1}}}, ..., {\overline B_{j_{l}}}) ].
\end{equation*}
Thus in order to see that (\ref{old17}) holds, in this ``events" setting, we only need to prove that 
 $$ E{{m-S}\choose k} {{n-T}\choose l}= \sum_{ i\in I_{k,m}} \sum_{ j\in I_{l,n}} P({\overline A_{i_{1}}}, ..., {\overline A_{i_{k}}}; {\overline B_{j_{1}}}, ..., {\overline B_{j_{l}}}).
 $$
 We can  use the method of indicators, as for example in page 292 of Hoppe and Seneta (2012) to do  this (Replace $M_{m}$ with $m-U_{m}$ and $n-V_{n}$ respectively there,  multiply the two items and then take expectation on both sides).
 
 To conclude this section, we point out that in univariate theory of sets, Boole's Inequality, using the notation (\ref{old1}):
$$P(S\ge 1)\le S_{1}$$
is the first upper Bonferroni Inequality. From (\ref{old2}), the right-hand side equals ${E{{S}\choose 1}}=ES$, to which the right-hand side of Gumbel's Inequality (\ref{old3}) reduces at $k=1$. Next, from (\ref{old15}) the first upper Bonferroni Inequality for two sets using notation (\ref{old5}), is:

 $$P(S\ge 1, T\ge 1)\le S_{1,1}=E(ST)$$
 from (\ref{old4}).

The proposed inequality (\ref{old22}) for general $(S,T)$ with $k=l=1$ on the right hand side gives for the bound
 
$$ \frac{E{\big\{} {\big\{} {{m}\choose 1} - {{m-S}\choose 1} {\big\}} {\big\{} {{n}\choose 1} - {{n-T}\choose 1} {\big\}} {\big\}} }   {{{m-1}\choose {1-1}} {{n-1}\choose {1-1}}}=E(ST),$$
which supports, in this sense, (\ref{old22}) as an appropriate bivariate binomial generalization of (\ref{old3}).

\end{document}